%% file: main.tex
\documentclass{article}
\usepackage{graphicx} 

\title{Planar 1-ended graphs can be periodically coloured}
\author{Luke Waite}
\date{}

\input{Preamble}

\begin{document}

\maketitle
\begin{abstract}
We conclude an investigation of Abrishami, Esperet, Giocanti, Hamman, Knappe and Möller studying the existence of periodic colourings of locally finite graphs. A colouring of a graph $\Gamma$ is \textit{periodic} if the resulting coloured graph has a finite number of orbits under its colour-preserving automorphisms, as such it is natural to consider those \textit{quasi-transitive graphs} with finite quotient. In the case that the graph is planar and has 1-end we prove that it always permits a periodic proper vertex colouring. This is shown by constructing isometry respecting embedded maps into the Euclidean and hyperbolic planes and leveraging known properties of Euclidean and hyperbolic isometry groups. Moreover, in the case that a graph is Euclidean we show that this can always be done in 5 colours.


\end{abstract}
\section{Introduction}\label{sec:intro}
The main purpose of this paper is to complete a recent investigation started by Abrishami, Esperet, Giocanti, Hamman, Knappe and Möller in \cite{abrishami2025periodiccoloringsorientationsinfinite}. Therein the authors look at when highly-symmetric graphs can be endowed with an additional highly-symmetric structure that respects some of the original symmetries. 
In particular, in this paper we shall focus on whether highly-symmetric \textit{planar} graphs permit the assignment of a highly-symmetric proper vertex colouring.
A highly-symmetric vertex colouring can be taken to mean one such that there are a finite number of orbits under the colour-respecting automorphisms of the resulting coloured graph; such a vertex colouring is called \textit{periodic}. One requirement for such a colouring is that the original graph has a finite number of vertex orbits under its automorphism group to begin with, it is natural to ask which of these \textit{quasi-transitive} graphs permit periodic colourings.
Abrishami and collaborators ask:

\begin{problem}[\cite{abrishami2025periodiccoloringsorientationsinfinite}, Problem 1.1\footnote{also raised as Problem 6.3 in \cite{ESPERETGIOLEG-DUT2024}.}]\label{q:allgraphspermitperiodic?}
    Do all locally-finite quasi-transitive graphs permit periodic colourings?
\end{problem}

It is not difficult to construct examples of these colourings; the 2-colouring of the Cayley graph of $\ZZ^2$ shown in Figure \ref{fig:Z_squared_periodic_colouring} is periodic. Moreover, any proper colouring of a square section of the graph can be extended to a periodic colouring of the whole graph by translation. In fact for any sufficiently `nice' finitely generated group, a periodic colouring of its Cayley graph may be found by colouring the finite quotient graph induced by the action of a finite index proper subgroup, and lifting with respect to the quotient map.

 However, in general they show quasi-transitivity is not sufficient for existence of such colourings, and hence prove the negative to Problem \ref{q:allgraphspermitperiodic?}
 
 The authors achieve this by constructing 1-ended (non-planar) counterexamples. Norin and Przytycki are additionally credited for finding counterexamples that are also Cayley graphs (and hence vertex transitive).
By restricting the hypotheses of the question, the authors also show that all locally-finite quasi-transitive graphs $\Gamma$ with bounded pathwidth do permit a periodic colouring, and hence proving a positive answer for all locally-finite, quasi-transitive graphs with 2 ends. 
In addition, Abrishami and collaborators construct several $\infty$-ended planar graphs that do not permit periodic colourings.
\begin{figure}
    \centering
    \includegraphics[width=0.25\linewidth]{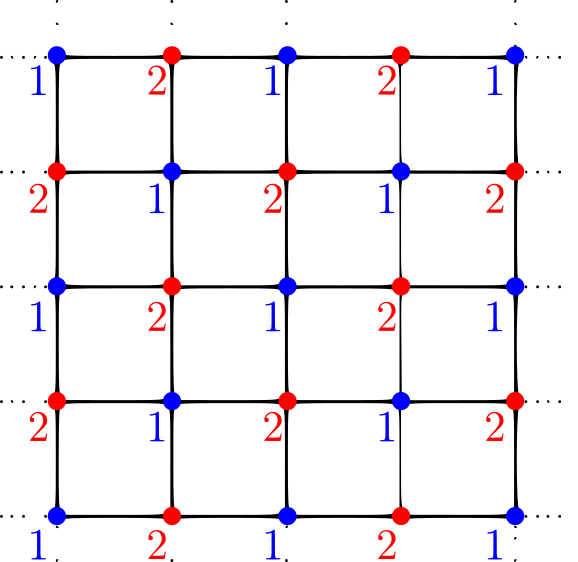}
    \caption{A 2-colouring of the Cayley graph of $\ZZ\times\ZZ$ that is $(2\ZZ\times 2\ZZ)$-periodic.}
    \label{fig:Z_squared_periodic_colouring}
\end{figure}
Following \cite[Proposition 2.1]{Babai97}, which tells us every infinite locally-finite quasi-transitive graph has 1, 2 or $\infty$-many ends, this leaves only one case outstanding.
\begin{problem}[\cite{abrishami2025periodiccoloringsorientationsinfinite}, Problem 6.6]
\label{q:without periodic c?}
Are there locally finite, quasi-transitive planar graphs with 1-end that do not a permit periodic colouring?
\end{problem}
Of these highly symmetric graphs, those that are planar are of particular interest as they naturally arise as periodic tilings of the planar geometries. As such they induce embeddings of their finite quotients onto well-defined quotient surfaces. In this paper we prove the answer to Problem \ref{q:without periodic c?} in the negative. That is:
\begin{theorem}\label{thm:main_intro}
    Let $\Gamma$ be a locally-finite, quasi-transitive planar graph with 1-end. Then $\Gamma$ permits a periodic proper vertex colouring.
\end{theorem}
Combined with the work in \cite{abrishami2025periodiccoloringsorientationsinfinite}, this  completes the discussion of periodic vertex colourings for locally-finite, quasi-transitive planar graphs. That is
\begin{corollary}
    Let $\Gamma$ be a locally-finite, quasi-transitive planar graph with finitely many ends. Then $\Gamma$ admits a periodic proper vertex colouring.
\end{corollary}



We close this introduction by giving a brief survey of extant results.
We shall make use of the following theorem of Babai, which asserts in the case that the graphs are 3-connected we may identify their automorphism groups with groups of Euclidean or hyperbolic plane isometries. 
\begin{theorem}[\cite{Babai97}, Theorem 4.2]\label{thm:Babai quasi-transitive}
    Let $\Gamma$ be a locally-finite 3-connected quasi-transitive planar graph with at most one end. Then $\Gamma$ has an embedding on a natural geometry such that all automorphisms of $\Gamma$ are induced by isometries of the geometry.
\end{theorem}
In a paper of Timár concerning colourings of random planar maps, a similar question to that of Problem \ref{q:without periodic c?} is asked (\cite{TIMÁR_2011}, Question 4.8) and a sketch argument is provided for graphs that embed into the Euclidean plane under Theorem \ref{thm:Babai quasi-transitive}. 
A crutch of Timár's argument is a powerful theorem of Thomassen concerning 5-colourability of certain maps on closed orientable surfaces.
\begin{theorem}[\cite{Thomassen}]\label{thm:Thomassen}
    Let $\Gamma$ be an embedded graph on a closed orientable surface of genus $g>0$ such that all non-contractible cycles contain at least $2^{14g+6}$ distinct edges. Then $\Gamma$ admits a proper 5-colouring.
\end{theorem}
As such, Timár's argument provides us with the following strengthening of the main result for those 3-connected graphs that embed into the Euclidean plane under Theorem \ref{thm:Babai quasi-transitive}.
\begin{theorem}\label{thm:Euc_5_col}
    Let $\Gamma$ be a 3-connected, locally-finite, quasi-transitive planar graph with 1-end that embeds into the Euclidean plane under Theorem \ref{thm:Babai quasi-transitive}. Then $\Gamma$ permits a periodic 5-colouring.
\end{theorem}
Moreover, in \S\ref{sec:not_3-con} when dealing with those graphs that are not 3-connected, we are able to extend this result to a larger class of graphs that embed into the Euclidean plane with the property that a finite index subgroup of the group of automorphisms is induced by Euclidean isometries.

We shall proceed by leveraging properties of Euclidean and hyperbolic planar isometry groups, in particular the existence of finite index torsion free subgroups. In the Euclidean case we reproduce Timár's argument. By finding torsion free subgroups, which in the Euclidean case are simply 2-generator translation groups, we may define an embedding of the quotient graph onto a sufficiently large torus and apply Theorem \ref{thm:Thomassen}. 
Using classical results on the subgroups of Fuchsian groups we make an analogous construction in the hyperbolic case, though due to the arbitrarily large potential genus of a hyperbolic quotient surface an application of Theorem \ref{thm:Thomassen} does not so easily follow. As such a more direct construction of a periodic colouring is applied. 

To deal with the graphs that are not 3-connected we shall make use of a known reduction technique 
to construct 3-connected topological subgraphs that behave well with respect to the automorphism group of the original graph.
We then build an embedding of the whole graph from the embedding of its 3-connected topological subgraph that is provided by Theorem \ref{thm:Babai quasi-transitive}. 

We first define relevant terms and prove some useful properties the group of automorphisms must satisfy when considered as a group of planar isometries. In \S\ref{sec:Euclid} a formal re-creation of Timár's proof of Theorem \ref{thm:Euc_5_col} is provided. In \S\ref{sec:Hyp}, using properties of hyperbolic isometries and a theorem of Edmonds, Ewing and Kulkarni \cite{Edmonds1982} on finite index torsion free subgroups of Fuchsian groups, we show that all 3-connected graphs that embed into the hyperbolic plane permit periodic colourings. In \S \ref{sec:not_3-con} we deal with those graphs that are not 3-connected, by the aforementioned reduction technique and embedding construction, which concludes the proof of Theorem \ref{thm:main_intro}. 
Finally, in \S\ref{sec:discussion} we discuss some consequences of this work in terms of periodic edge-colourings and orientations, as well as some open problems that naturally arise.


\section{Preliminaries and General properties}\label{sec:Prelim_Gen}
\textbf{Graphs and maps:} We define a graph to be pair $\Gamma=(V,E)$, where $V$ is the vertex set, and $E$ the set of edges represented by unordered pairs $\{u,v\}\subset V$ so that $u$ and $v$ are joined by an edge. In cases where it is useful to specify the edge and vertex sets of a specific graph $\Gamma$ we will write $V(\Gamma)=V$ and $E(\Gamma)=E$. A graph $\Gamma$ is said to be infinite if its vertex set is infinite. A graph is said to be \textit{locally-finite} if the degree at each vertex is finite.
All graphs considered will be infinite, locally-finite, undirected, and without multiple edges or loops.\\
A \textit{map} $M$ is a graph embedded in a surface $X$ so that the vertices $v\in V(M)$ are distinct points of $X$, the edges $\{u,v\}\in E(M)$ are curves $\alpha_{u,v}:[0,1]\rightarrow X$ such that $\alpha_{u_1,u_2}$ and $\alpha_{v_1,v_2}$ intersect at $x\in V(M)$ if and only if $u_i=v_j=x$ for some $i,j\in\{1,2\}$ and, if we cut the surface along the set of all such curves what remains ($X-M$) is a disjoint union of connected components called \textit{faces}, all of which are homeomorphic to open discs.
A graph $\Gamma$ is said to be \textit{planar} if there exists a map $M\subset \EE^2$ and an embedding $\varphi:\Gamma\rightarrow \EE^2$ into the Euclidean plane so that $\varphi(\Gamma)=M$.
It will be useful to distinguish between a combinatorial graph $\Gamma$ and a corresponding embedded map $M$ in a surface $X$.\\
A \textit{path} in a graph $\Gamma$ from a vertex $v_0$ (called its initial vertex) to vertex $v_m$ (its terminal vertex) is a sequence of incident edges $(v_0,v_1),(v_1,v_2),\ldots,(v_{m-1},v_m)$ such that every vertex is visited only once, that is $v_i\neq v_j$ for all $i,j\leq m,$ $i\neq j$.
A path is called a \textit{cycle} if its initial and terminal vertices coincide.
A \textit{ray} $r$ in an infinite graph $\Gamma$ is an infinite one-way path. We say that two rays $r_1,r_2$ are \textit{end-equivalent} if and only if there exists some third ray $r_3$ that contains infinitely many vertices of both $r_1$ and $r_2$. Equivalently we may say that $r_1$ and $r_2$ are end-equivalent if and only if there are infinitely many disjoint paths connecting vertices of $r_1$ to vertices of $r_2$. 
An equivalence class of rays in $\Gamma$ is called an \textit{end} of $\Gamma$, and a graph with exactly $k$ ends is said to be $k$\textit{-ended}. In practice, given any infinite sequence of finite subgraphs $\Sigma_1,\Sigma_2,\ldots\subset \Gamma$ such that $V(\Sigma_{i+1})\subsetneq V(\Sigma_i)$, the number of ends of $\Gamma$ can be found as the limit of the number of infinite connected components of $\Gamma-\Sigma_i$ as $i\rightarrow\infty$. 
A graph is said to be \textit{connected} if there exists a path between any pair of distinct vertices. Throughout, unless otherwise stated all graphs will be assumed to be connected. If the removal of any set of less than $q$ vertices leaves $\Gamma$ connected, then $\Gamma$ is said to be $q$\textit{-connected}. 
A map is said to be locally-finite, $k$-ended or $q$-connected, if its respective graph is.\\
For graphs $\Gamma_1,\Gamma_2$ a \textit{graph isomorphism} is a map $f:V(\Gamma_1)\rightarrow V(\Gamma_2)$ that respects edge adjacency, that is for all $\{u,v\}\in E(\Gamma_1)$ we have that $(f(u),f(v))\in E(\Gamma_2)$.
A \textit{graph automorphism} is a graph isomorphism where $\Gamma_1=\Gamma_2$. The set of all graph automorphisms for a graph $\Gamma$ forms a group under composition, which we will write $\text{Aut}(\Gamma)$. 
For a map $M$ in a surface $X$, a \textit{map automorphism} is a homeomorphism $f:X\rightarrow X$ such that $f(V(M))=V(M)$ and the induced action on $V(M)$ is a graph automorphism.

\textbf{Groups and colourings:} Given a group $G$ acting on a graph $\Gamma$, we say that $G$ acts \textit{transitively} if for all $u,v\in V(\Gamma)$ there exists $g\in G$ so that $gu=v$. An \textit{orbit} of a vertex $v\in V(\Gamma)$, is defined as the set $Gv:=\{gx:g\in G\}$.
In other words a group acts transitively if there is a single orbit of the vertices under the action of $G$. A group $G$ is said to act \textit{quasi-transitively} on $\Gamma$ if the action of $G$ partitions $V(\Gamma)$ into finitely many orbits. If $\text{Aut}(\Gamma)$ acts transitively on $\Gamma$ then we say that the graph $\Gamma$ is \textit{vertex transitive}, or in the case that $\text{Aut}(\Gamma)$ acts quasi-transitively we say that $\Gamma$ \textit{is quasi-transitive} (also called {almost transitive} by some authors).
\begin{remark}
   It follows that any graph that is both locally-finite and quasi-transitive has a vertex set that is at most countably infinite. 
\end{remark}
A \textit{(proper) vertex colouring} $\mathcal{C}$ is a function $\mathcal{C}:V(\Gamma)\rightarrow \{1,\ldots,m\}$, $m<\infty$, such that for all $\{u,v\}\in E(\Gamma)$, $\mathcal{C}(u)\neq\mathcal{C}(v)$. 
Given a (non-traivial) group $H$ acting on a graph $\Gamma$ by graph automorphisms, a (proper) vertex colouring $\mathcal{C}$ is said to be ($H$-)\textit{periodic} if the action on the graph preserves the colouring. That is, for all $h\in H$, $x\in V$
$$\mathcal{C}(hx)=\mathcal{C}(x).$$
In other words, a colouring $\mathcal{C}$ is periodic if the corresponding coloured graph is quasi-transitive.

So that we can apply Babai's Theorem \ref{thm:Babai quasi-transitive}, we first restrict our attention to those graphs that are are 3-connected. Let $\Gamma$ be a locally-finite, 3-connected, quasi-transitive planar graph with 1-end. Let $\varphi:\Gamma\rightarrow \mathcal{X}$ be the embedding inferred by Theorem \ref{thm:Babai quasi-transitive}, and let $M:=\varphi(\Gamma)$ be the embedded map on $\mathcal{X}$, for $\mathcal{X}=\EE^2$ or $\HH^2$. Then  $\text{Aut}(\Gamma)\cong\text{Aut}(M)=:G\leq \text{Isom}(\mathcal{X})$. We may wish to abuse notation somewhat and also refer to $\text{Aut}(\Gamma)$ by $G$.
An immediate consequence is that, if $\{u,v\},\{x,y\}\in E(\Gamma)$ are in the same $\text{Aut}(\Gamma)$ orbit, then the corresponding edge curves in $M$, $\alpha_{u,v}$ and $\alpha_{x,y}$ are isometric.
Inline with the terminology of \cite{Babai97} we define the following;

\begin{definition}
    Suppose a graph $\Gamma$ embeds into a manifold $X$, with embedded map $M\subset X$. $X$ is said to be \textit{plane} with respect to $M$ if every cycle in $M$ partitions $X$ in two parts. An embedding $M$ in $X$ of a graph $\Gamma$ is a \textit{plane embedding} if the manifold $X$ is plane with respect to $M$.
\end{definition}
Both $\EE^2$ and $\HH^2$ can be considered as plane manifolds, and the embedding inferred by Theorem \ref{thm:Babai quasi-transitive} a plane embedding.
Recall in a topological space $X$, a family of subsets $\{S_i\}_i$ is said to be a \textit{locally finite collection} if for any compact set $K\subset X$, $K$ has non-empty intersection with at most finitely many $S_i$. For a map $M$ on a closed surface $X$, we say that $V(M)$ is a locally finite collection \textit{of vertices} if the family of singleton subsets of $V(M)$ is a locally finite collection.

\begin{lemma}\label{lem:locally finite collection}
    Let $\Gamma$ be a locally-finite, 3-connected, quasi-transitive planar graph with 1-end that embeds as a map $M$ on a natural geometry $\mathcal{X}$ ($\mathcal{X}=\EE^2$ or $\HH^2$) under Theorem \ref{thm:Babai quasi-transitive}. Then $V(M)$ is a locally finite collection of vertices.
\end{lemma}
\begin{proof}
By the construction of Babai's embedding it follows that $M$ contains no faces with infinitely many sides \cite[Lemma 2.2]{Babai97}. Moreover since $M$ is quasi-transitive and $\text{Aut}(M)$ acts by isometries of $\mathcal{X}$, the set of distances between adjacent vertices $\mathcal{E}:=\{\rho_\mathcal{X}(u,v):\{u,v\}\in E(M)\}$ is finite, and the set of lengths of the edge curves $\alpha_{u,v}$ is also finite. Hence we may conclude all faces of $M$ have compact closure. 
 For a simple closed loop $C:[0,1]\rightarrow \mathcal{X}$, let $\overline{C}\subset\mathcal{X}$ be the induced compact set such that $\partial\overline{C}=C$.
Observe that in $\mathcal{X}$, any compact set $K$ is contained in a union of at most finitely many compact sets of the form $\overline{C}$.
Hence it is sufficient to show that for any simple closed loop $C:[0,1]\rightarrow \mathcal{X}$, the induced compact set $\overline{C}$ contains at most finitely many vertices of $M$. \\
First suppose that $C$ coincides with a cycle in $M$. Then as $M$ has 1-end the induced subgraph on $V(M)-V(C)$ consists of two connected components $Q_1,Q_2$ such that $V(Q_1)$ is finite and $V(Q_2)$ is infinite. Since $\overline{C}$ is compact, it has finite area, so that $\overline{C}$ can contain only a finite natural number of faces of $M$. As the boundary of each face consists of finitely many vertices, we have that
$$V(M)\cap\overline{C}=V(C)\cup V(Q_1)<\infty.$$
Let $C$ now be an arbitrary loop in $\mathcal{X}$, and suppose that we have at least one vertex $x\in V(M)\cap \overline{C}$. Since $M$ has 1-end, and $\mathcal{E}$ is finite, it follows that for any given vertex $x$, and some $r>0$, we can find a cycle $C'$ in $M$ such that $\overline{C'}$ contains the open ball $ B(x,r)$. Then we may choose $r$ sufficiently large so that $\overline{C}\subset B(x,r)\subset\overline{C'}$, and the result follows.\\
\end{proof}

For the purposes of a cleaner argument when dealing with graphs that are not 3-connected, we proceed with only the assumption that we have group of isometries of $\mathcal{X}$ acting quasi-transitively on $M$.

\begin{lemma}\label{lem:discrete group}
    Let $M$ be a locally-finite map with 1-end in $\mathcal{X}\;(=\EE^2,\HH^2)$, such that $V(M)$ is a locally finite collection of vertices. Suppose that $G\leq \text{Isom}(\mathcal{X})$ acts quasi-transitively on $M$. Then $G$ is a discrete group of isometries of $\mathcal{X}$.
\end{lemma}
\begin{proof}
    Suppose that $G$ is not discrete. Then there exists a sequence $\{\gamma_n\}_{n=1}^\infty\subset G-\{\text{Id}\}$ such that for all $x\in\mathcal{X}$ we have that $\rho_\mathcal{X}(\gamma_n(x),x)\rightarrow 0$ as $n\rightarrow \infty$ (where $\rho_\mathcal{X}(\cdot,\cdot)$ denotes the relevant Euclidean of hyperbolic metric on $\mathcal{X}$). However as $M$ is $G$-invariant, for a vertex $v\in V(M)$ this implies that there exists a sequence of vertices $v,\gamma_1v,\gamma_2v,\ldots\subset V(M)$ with $\rho_\mathcal{X}(v,\gamma_iv)\rightarrow 0$, contradicting that $V(M)$ is a locally finite collection of vertices.\\    
\end{proof}

\section{Graphs that embed in the Euclidean plane}\label{sec:Euclid}
We now consider only the case where $\mathcal{X}=\EE^2$. Discrete groups of isometries of the Euclidean plane are well studied and classified. A discrete group of isometries such that there exist two linearly independent translations is called a \textit{wallpaper group}, of which it is known there are 17 up to isomorphism. A discrete group of isometries with a single translation up to linear dependence is known as a \textit{Frieze group}, of which there are 7 up to isomorphism (see \cite{conway2008symmetries} for a full exposition on these groups).

For a discrete group of isometries $\mathcal{G}$ acting on a metric space $X$, a \textit{fundamental domain} $F\subset X$ is an open connected set so that;
\begin{enumerate}[label=(\roman*)]
    \item $\bigcup_{g\in \mathcal{G}}g \Bar{F}=X$, and
    \item for $\mathcal{G}\ni g\neq 1$, ${F}\cap gF=\emptyset$.
\end{enumerate} 
The set $\partial F = \Bar{F}-F$ is called the boundary of $F$.
In addition, it is well known that for any point $x\in X$, the intersection $\partial F\cap Gx$ is empty or finite.

A Euclidean isometry $t\in \text{Isom}(\EE^2)$ that acts without fixed points is called a \textit{translation}, and is completely determined by its direction and translation length. A group of isometries $T\leq \text{Isom}(\EE^2)$ that, with the exception of the identity, consists entirely of translations is called a \textit{translation group}.
It is a fact of linear algebra that any translation group acting on the Euclidean plane can have at most two linearly independent generating translations. Let $T=\langle t_1,t_2\rangle$ be a two-dimensional translation group generated by linearly independent translations $t_1$ and $t_2$. Fix some $x\in \EE^2$, and define line segments $s_i:[0,1]\rightarrow \EE^2$, such that $s_i(0)=x$ and $s_i(1)=t_i(x)$ for $i\in\{1,2\}$. Then each $s_i$ is parallel to the action of the respective $t_i$, and has length equal to the translation length of $t_i$.
It then follows that the interior of any parallelogram defined by line segments $s_1,s_2$ is a fundamental domain for $T$.

\begin{lemma}\label{lem:wallpaper}
  Let $M$ be a locally-finite map with 1-end in $\EE^2$ such that $V(M)$ is a locally finite collection of vertices, and suppose $G\leq \text{Isom}(\EE^2)$ acts quasi-transitively on $M$. Then $G$ is a wallpaper group.
\end{lemma}
\begin{proof}
    It suffices to show that there exist linearly independent translations $t_1,t_2\in G$. We first suppose that $G$ contains no translations. Then all orbits $Gx$, $x\in\EE^2$, are finite. Since $G$ acts quasi-transitively this implies that $V(M)$ is finite, contradicting that $M$ has 1-end.\\
    Suppose that there is a single translation $t\in G$ up to linear dependence. Without loss of generality suppose that $t$ attains the minimum translation distance. Then $G$ is a Frieze group and has a finite index infinite-cyclic translation subgroup $T=\langle t\rangle$. Then $T$ has a fundamental domain $F$ that is the interior of two parallel lines $\ell_1,\ell_2\subset \EE^2$. Let $d=\text{dist}(\ell_1,\ell_2)$, this corresponds with the translation distance of $t$. 
    The set of distances between adjacent vertices $\mathcal{E}:=\{|u-v|:\{u,v\}\in E(\Gamma)\}$ must be finite so there exists a maximum $\max\{\mathcal{E}\}$.
    Choose $n>0$ sufficiently large that the translation distance of $nt$ is at least $\max\{\mathcal{E}\}$. Then $T_n=\langle nt\rangle$ is index $n$ in $T$ and has a fundamental domain $F_n$  so that $\Bar{F}_n$  consists of the connected union of $n$ copies of $\bar{F}$. It follows that $T_n$ is finite index in $G$. Since $G$ acts quasi-transitively the set $\bar{F}\cap V(M)$ is finite, and hence $\bar{F}_n\cap V(M)$ is finite. Then removing $\bar{F}_n$ from the plane leaves the disjoint union
     $$\EE^2-F_n=P_1\sqcup P_2,$$
    where $P_1$ and $P_1$ are infinite connected components of the plane. Moreover by construction we have that if $u\in V(M)\cap P_1$ and $v\in V(M)\cap P_2$ then they cannot be connected by an edge of $M$.
    Since the orbit of $F_n$ under $T_n$ is bi-infinite, it follows that the induced subgraphs on $P_1\cap V(M)$ and $P_2\cap V(M)$ are infinite and connected. Then $F_n\cap V(M)$ serves as a finite cut set leaving two infinite connected components, contradicting that $M$ has 1-end.
    \end{proof}

\begin{theorem}\label{thm:main_Euclid}
    Let $M$ be a locally-finite map in $\EE^2$ with 1-end, and such that $V(M)$ is a locally finite collection of vertices. Suppose $G\leq \text{Isom}(\EE^2)$ acts quasi-transitively on $M$. Then there exists a vertex colouring $\mathcal{C}:V(M)\rightarrow \{1,\ldots,5\}$ such that $\mathcal{C}$ is $T$-invariant, for some finite index subgroup $T\leq G$.
\end{theorem}
\begin{proof}
  By the above we have that $G$ is a wallpaper group. Let $T'\subseteq G$ be the set of all translations of $G$. Then it follows that $T'$ forms a translation subgroup of $G$. For any point $x\in\EE^2$, the orbit $T'x$ forms a lattice in the plane, which will take one of five forms known as the 2-dimensional Brevais lattices \cite{Schattschneider}. It follows that we may construct a fundamental domain $F'$ for $T'$ by drawing a minimal area parallelogram $\overline{F'}=\square ABCD$ whose vertices $A,B,C,D\in T'x$ are points in the lattice, and taking $F'$ to be its interior. Furthermore, if $t_1,t_2\in T'$ are the translations along vectors defined by adjacent sides of $F'$ (say $\overrightarrow{AB},\overrightarrow{AC}$ respectively), then $\langle t_1,t_2\rangle = T'$. Since $\overline{F'}$ is compact it contains at most finitely many vertices of $\Gamma$.
  Then the quotient surface $\EE^2/T'$ is a torus, and we have an induced a finite map $M/T'$ on $\EE^2/T'$ as the image of $M$ under the quotient function $\pi_{T'}:\EE^2\rightarrow \EE^2/T'$.
  In order to satisfy the hypotheses of Theorem \ref{thm:Thomassen} we require that the minimum number of edges in a path $\gamma$ that intersects opposite sides of a fundamental parallelogram is at least $2^{20}$. Choose integers $A,B$ sufficiently large that all fundamental parallelograms $P$ for $T=\langle At_1,Bt_2\rangle$, this holds. Then $T$ is an finite index subgroup of $T'$, such that under the quotient map $\pi_T$, $M/T$ is a finite map on the torus $\EE^2/T$ satisfying the assumptions of Thomassen's theorem. Hence, there exists a proper 5-colouring $\mathcal{C}:V(M/T)\rightarrow\{1,\ldots,5\}$. It follows that $\mathcal{C}$ lifts to 5-colouring $\widetilde{\mathcal{C}}:V(M)\rightarrow \{1,\ldots,5\}$ that is $T$-invariant.
\end{proof}


\section{Graphs that embed in the hyperbolic plane}\label{sec:Hyp}
 We proceed in the Poincaré disc model $\mathbb{D}$ for the hyperbolic plane, where $\mathbb{D}:=\{z\in\CC:|z|< 1\}$. Let $\mathbb{D}$ denote the interior of the disc, $\partial \mathbb{D}$ denote the principal circle `at infinity' and $\overline{\mathbb{D}}=\mathbb{D}\cup\partial\mathbb{D}$ its closure with the \textit{boundary at infinity}.
 We shall write $\rho_\mathbb{D}(\cdot,\cdot)$ to denote the appropriate hyperbolic metric on $\mathbb{D}$. A degree of familiarity with elementary hyperbolic geometry is assumed here, for an introductory approach see \cite{anderson2005hyperbolic}. 

For a discrete group of isometries $G\leq \text{Isom}(\HH^2)$, and a point $\alpha\in\mathbb{D}$ that is not a fixed point for any element of $G$, a \textit{Dirichlet domain} $D(\alpha)$ is a fundamental domain such that $D(\alpha)=\bigcap_{\gamma\in G}\{z\in\mathbb{D}:\rho_\mathbb{D}(z,\alpha)< \rho_\mathbb{D}(z,\gamma(\alpha))\}$. A Dirichlet domain is a convex hyperbolic polygon whose boundary consists of geodesic arcs, these are referred to as the \textit{sides} of $D(\alpha)$, and the point where two incident sides meet as a \textit{vertex} of $D(\alpha)$. 
In the case that a Dirichlet domain is an unbounded set, we may consider its closure `at infinity' $\overline{D(\alpha)}$  in $\overline{\mathbb{D}}$. 
Each point where a side of $\overline{D(\alpha)}$ intersects $\partial\mathbb{D}$ is called an \textit{ideal} vertex of $D(\alpha)$, and each section of the boundary $\partial\mathbb{D}$ contained in $\overline{D(\alpha)}$ an \textit{ideal} side.
For a point $x\in \mathbb{D}$ and $r>0$ let $C(x,r)$ denote the hyperbolic circle with centre $x$ and hyperbolic radius $r$. For a connected subset $A\subseteq{\mathbb{D}}$ let $\mu(A)$ denote its hyperbolic area.

An orientation-preserving discrete group of isometries $G^+\leq \text{Isom}^+(\HH^2)$ is called a Fuchsian group.
Every discrete group of hyperbolic isometries $G\leq \text{Isom}(\HH^2)$ has an \textit{orientation-preserving} subgroup $ G^+\leq G $ of index at most 2 (index 1 if $G$ itself is orientation-preserving). It follows immediately that if a group $G$ acts quasi-transitively on a map $M$, then $G^+$ also acts quasi-transitively on $M$ (with at most twice as many orbits). Moreover, if $F$ is a fundamental domain for $G$ and $F^+$ a fundamental domain for $G^+$, so that $[G:G^+]=k\in\{1,2\}$, we have that $\mu(F^+)=k\cdot\mu(F)$ (\cite{Ratcliffe_manifolds}, 6.7.3).
Fuchsian groups are well studied, as such it will often be useful to consider the orientation-preserving subgroup of a given group of isometries. 

\begin{lemma}\label{lem:hyp_fin_area}
    Let $M$ be a locally-finite map in the hyperbolic plane with 1-end such that $V(M)$ is a locally finite collection of vertices. Suppose that  $G\leq\text{Isom}(\HH^2)$ acts quasi-transitively on $M$. Then $\mu(F)<\infty$ for all fundamental domains $F\subset\overline{\mathbb{D}}$ for $G$.
\end{lemma}
\begin{proof}
    We assume that $G$ is Fuchsian, as if it is not then $G$ has an index-2 subgroup $G^+$ that is Fuchsian, and for all fundamental domains $F$ and $F^+$ for $G$ and $G^+$ respectively, we have that $\mu(F)=\mu(F^+)/2$.
  It is known that a Fuchsian group is finitely generated if it is cocompact. Furthermore, a Fuchsian group is cocompact if and only if any fundamental domain has compact closure (\cite{Fuchsian_Katok}, 4.2.3). 
  Suppose towards a contradiction that for all fundamental domains $F$ for $G$ we have that $\mu(F)=\infty$. Then it follows that $G$ is a Fuchsian group \textit{of the second kind}, and hence all fundamental domains $F$ contain a closed half space (\cite{Ratcliffe_manifolds}, 12.1.12).\\
  Fix a fundamental domain $F$. For a given group of isometries it is known that the area of all fundamental domains are equal. Hence, without loss of generality we suppose that $F$ is a Dirichlet region $D(\alpha)$ for some ordinary point $\alpha\in \mathbb{D}$. Let $H_F$ be a maximal open half space contained in $F$, in the sense that there are no further open half spaces in $F$ that contain $H_F$. Then let $\ell_F=\partial H_F$ be the boundary geodesic of $H_F$.
  
  In order that $H_F$ be contained within $F$, we must have that the endpoints $e_1,e_2\in\partial \mathbb{D}$ of $\ell_F$ must lie in a connected component of $\overline{F}\cap\partial\mathbb{D}$. Hence  we may conclude that $F$ has at least one ideal side in $\overline{F}\cap\partial \mathbb{D}$. Moreover, since $H_F$ is maximal $e_1$ and $e_2$ must be the ideal vertices of $F$ at the endpoints of this ideal side. 
  
  Let $x\in\overline{F}\cap\partial\mathbb{D}$ be the midpoint on the boundary between $e_1$ and $e_2$ that lies in $\overline{F}$. 
  We define a geodesic ray $\gamma:[0,\infty)\rightarrow \overline{\mathbb{D}}$ so that $\gamma(0)=\alpha,\gamma(\infty)=x$ and $\rho_{\mathbb{D}}(\alpha,\gamma(t))=t$ for all $t>0$. Since $G$ acts on $M$ is quasi-transitively it follows that there exists an $r>0$ such that the perpendicular bisector $\ell_r$ to $\gamma$ at $\gamma(r)$, defines a closed half-space $H_r\subseteq H_F$, such that $H_r\cap V(M)=\emptyset $.
  
Given that $F$ contains the maximal half space $H_F$, there are three possibilities for the structure of $F-H_F$. Either there is at least one other distinct maximal half space $H_F'\subset F$, or $H_F$ is unique and $F$ has at least one bounded side, or $H_F$ is unique and $F$ has no bounded sides. We proceed by showing that in all cases the half space $H_F$ forces the vertex set $V(M)$ to be `pinched' between two unbounded empty regions that contain no vertices. This implies in some sense that the `width' of $M$ is bounded and contradicts the assumption that $M$ has 1-end.
  
  \textit{Case 1:} $F$ contains a second maximal half-space $H_F'$ such that $H_F\cap H_F'=\emptyset$.\\
  Let $\ell_F'$ be the geodesic so that $\ell_F'=\partial H_F'$, and let $x'\in\partial \mathbb{D}$ be the midpoint of the endpoints of $\ell_F'$. Define geodesic arc $\gamma':[0,\infty)\rightarrow\overline{\mathbb{D}}$ such that $\gamma'(0)=\alpha$, $\gamma'(\infty)=x'$ and so that $\rho_\mathbb{D}(\alpha,\gamma'(t))=t$ for all $t>0$. As above since $G$ acts $M$ quasi-transitively there exists $s>0$ such that the perpendicular bisector $\ell_s$ of $\gamma'$ at $\gamma'(s)$, defines a closed half-space $H_s\subseteq H_F'$ such that $H_s\cap V(M)=\emptyset$ (see Figure \ref{fig:finite_area_case_1}). Since $G$ acts on $M$ quasi-transitively by isometries, the set of edge lengths $\mathcal{E}:=\{\rho_\mathbb{D}\{u,v\}:\{u,v\}\in E(M)\}$ is finite. For $R>\max\{r,s\}$ let $\{z_j,z_j'\}=C(\alpha,R)\cap H_j$ for $j=r,s$. Then choose $R>\max\{r,s\}$ sufficiently large so that 
  $$\min\{\rho_\mathbb{D}(z_j,z_j'):j=r,s\}>\max\{\mathcal{E}\}.$$
 Then it follows that $\mathbb{D}-(B(\alpha,R)\cup H_r\cup H_s)$ is a disjoint union of two connected components $P_1$ and $P_2$ such that $M\cap P_i$ is an infinite connected graph for $i=1,2$ and if $u\in P_1\cap V(M)$, $v\in P_2\cap V(M)$ then they cannot be connected by an edge of $M$. Moreover, as the closure of $B(\alpha,R)$ is compact it can contain at most finitely many vertices of $M$, and $H_r$ and $H_s$ contain no vertices of $M$, so that $V(M)\cap(B(\alpha,R)\cup H_r\cup H_s)$ is a finite cutset of $M$ that leaves two infinite connected components. This contradicts that $M$ is 1-ended, hence $F$ must have finite area.

 \textit{Case 2:} $H_F$ is a unique maximal half-space, and the only sides of $F$ are those incident with $e_1$ and $e_2$ on $\partial\mathbb{D}$.\\
  An immediate consequence of $H_F$ being unique is that $F$ has only one ideal side.
Let $S_1,S_2$ be the non-ideal sides of $F$, and a vertex $y\in\overline{\mathbb{D}}$ such that $S_1\cap S_2=\{y\}$. Since $F$ is a Dirichlet region there exists a \textit{side-pairing} element $\tau\in G$ such that $\tau(S_1)=S_2$. Moreover, we then must have that $G=\langle\tau\rangle$ (\cite{Fuchsian_Katok}, 3.5.4). If $y\in\mathbb{D}$ then $\tau$ is an elliptic element. Since $M$ is quasi-transitive it follows that there exists $r>0$ such that $F\cap V(M)=B(y,r)\cap V(M)$ and is finite. Furthermore since $B(y,r)$ is $\tau$-invariant,
 and hence $G$-invariant so that $B(y,r)\cap V(M)=\bigcup_{g\in G}gF\cap V(M)=V(M)$. Since the closed ball
$\overline{B(y,r)}$ is compact this implies that $V(M)$ is finite contradicting that $M$ has 1-end.\\
If $y\in \partial \mathbb{D}$, then $\tau$ is a parabolic element. Let $\Sigma:[0,\infty)\rightarrow\overline{\mathbb{D}}$ be the geodesic path such that $\Sigma(0)=\alpha,\Sigma(\infty)=y$ and $\rho_\mathbb{D}(\alpha,\Sigma(t))=t$ for all $t>0$. Then since $F\cap V(M)$ is finite there exists an $r>0$ such that $B(\alpha,r)\cap F\cap V(M)=F\cap V(M)$. Let $\omega(t)$ denote the horocycle based at $y$ and with intersection $\omega(t)\cap\mathbb{D}\cap \Sigma=\Sigma(t)$. Since $\omega(t)$ is $\tau$-invariant for all $t>0$, hence $G$-invariant, it follows that the closed horocyclic region $\mathcal{H}_r$ defined by $\omega(r)$ is such that $\mathcal{H}_r\cap V(M)=\emptyset$ (see Figure \ref{fig:finite_area_case_2}). For $R>r$ let $\{z,z'\}=C(\alpha,R)\cap H_r$ and $\{\zeta,\zeta'\}=C(\alpha,R)\cap \mathcal{H}_r$. Then choose $R>r$ sufficiently large so that 
$$\min\{\rho_\mathbb{D}(z,z'),\rho_\mathbb{D}(\zeta,\zeta')\}>\max\{\mathcal{E}\}.$$
 Then it follows that $\mathbb{D}-(B(\alpha,R)\cup H_r\cup \mathcal{H}_r)$ is a disjoint union of two connected components $P_1$ and $P_2$ such that $M\cap P_i$ is an infinite connected graph for $i=1,2$ and if $u\in P_1\cap V(M)$, $v\in P_2\cap V(M)$ they cannot be connected by an edge of $M$. Moreover, as the closure of $B(\alpha,R)$ is compact it can contain at most finitely many vertices of $M$, and $H_r$ and $\mathcal{H}_r$ contain no vertices of $M$, so that $V(M)\cap(B(\alpha,R)\cup H_r\cup H_s)$ is a finite cutset of $M$ that leaves two infinite connected components. This contradicts that $M$ is 1-ended, hence $F$ must have finite area.

 \begin{figure}[ht]
     \centering
    \begin{minipage}{0.5\textwidth}
       \centering
    \includegraphics[width=0.7\textwidth]{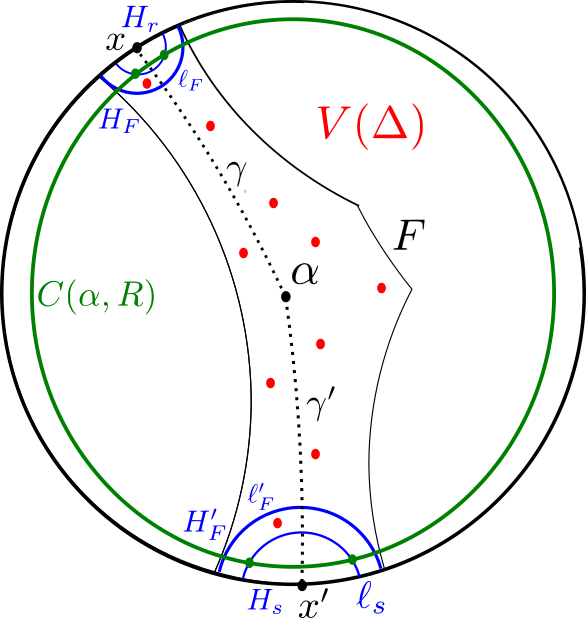}
    \caption{$F$ contains a second disjoint maximal half- \\ space $H_F'$.}
    \label{fig:finite_area_case_1}
    \end{minipage}%
    \begin{minipage}{0.5\textwidth}
    \centering
    \includegraphics[width=0.7\textwidth]{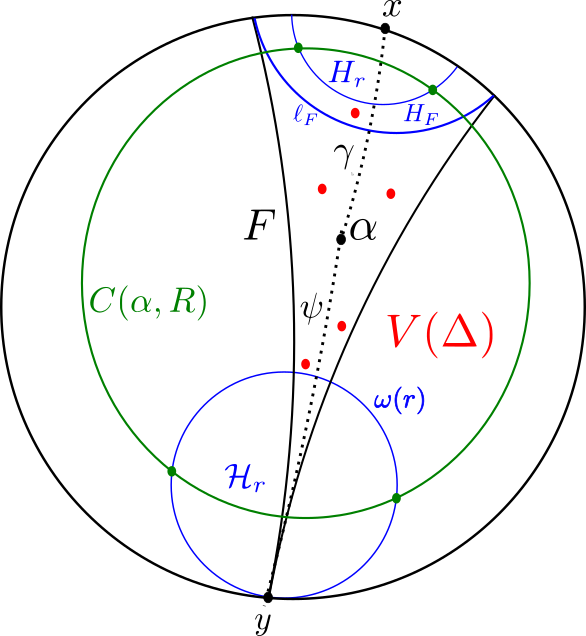}
    \caption{$H_F$ is unique and the only sides of $F$ intersect $\ell_F$.}
    \label{fig:finite_area_case_2}
    \end{minipage}
\end{figure}
 
\textit{Case 3:} $H_F$ is a unique maximal half-space, and there is at least one side $S$ of $F$ such that $S\cap\partial\mathbb{D\cap\ell_F}=\emptyset$, \\
 Then since $F$ is a Dirichlet region, it follows that there is another side $S'$ paired with $S$ so  that there exists $g\in G$ with $gS=S'$ and $F\cap gF=S'$.
  Consider the set $E=F\cup gF$. Then $E$ is a convex connected region, $E\cap V(M)$ is finite and $E$ contains distinct half-spaces $H_F$ and $gH_F$. Choose $\beta\in S'$ such that it lies in the interior of $S'$ (not a vertex of $F$), hence $\beta\in \mathring{E}$, and define geodesic arcs $\Sigma,\Sigma':[0,\infty)\rightarrow \overline{\mathbb{D}}$ so that $\Sigma(0)=\beta=\Sigma'(0)$, $\Sigma(\infty)=x$, $\Sigma'(\infty)=gx$ and $\rho_\mathbb{D}(\beta,\Sigma(t))=\rho_\mathbb{D}(\beta,\Sigma'(t))=t$ for all $t>0$. Let $r,s>0$ be such that the perpendicular bisectors of $\Sigma,\Sigma'$ at $\Sigma(r)$ and $\Sigma'(s)$ respectively define closed half-spaces $H_r$ and $H_s$ in $E$ so that $H_r\cap V(M)=\emptyset$ and $H_s\cap V(M)=\emptyset$ (see Figure \ref{fig:finite_are_case_3}). As before for $R>r,s$ let $\{z_j,z_j'\}=H_j\cap C(\beta,R)$ for $j=r,s$. Then choose $R>\max\{\mathcal{E}\}$ sufficiently large so that 
  $$\min\{\rho_\mathbb{D}(z_j,z_j'):j=r,s\}>\max\{\mathcal{E}\}.$$
  Then $\mathbb{D}-(B(\beta,R)\cup H_r\cup H_s)$ consists of two infinite connected components of $\mathbb{D}$ each containing infinite connected induced subgraphs of $M$. As before we then have that $V(M)\cap(B(\beta,R)\cup H_r\cup H_s)$ serves as a finite cut set, contradicting that $M$ has 1-end. \\

\end{proof}
\begin{figure}
     \centering
    \begin{minipage}{0.5\textwidth}
       \centering
    \includegraphics[width=0.7\linewidth]{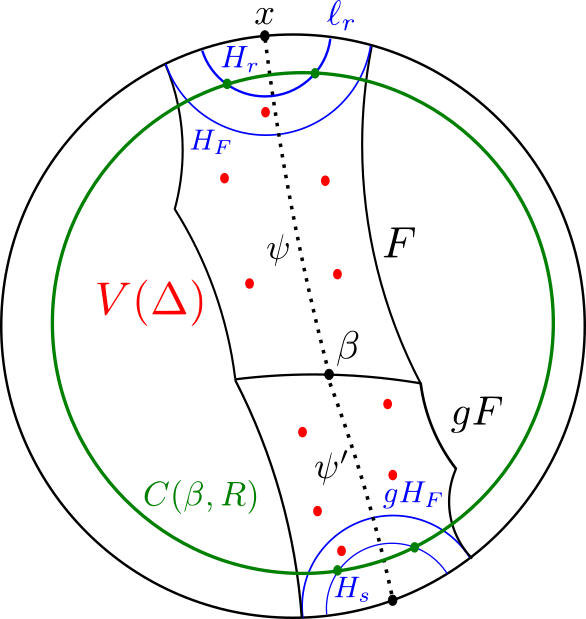}
    \caption{$H_F$ is unique and there is at least one side $S$ not incident with $\ell_F$.}
    \label{fig:finite_are_case_3}
    \end{minipage}%
    \begin{minipage}{0.5\textwidth}
    \centering
      \includegraphics[width=0.7\linewidth]{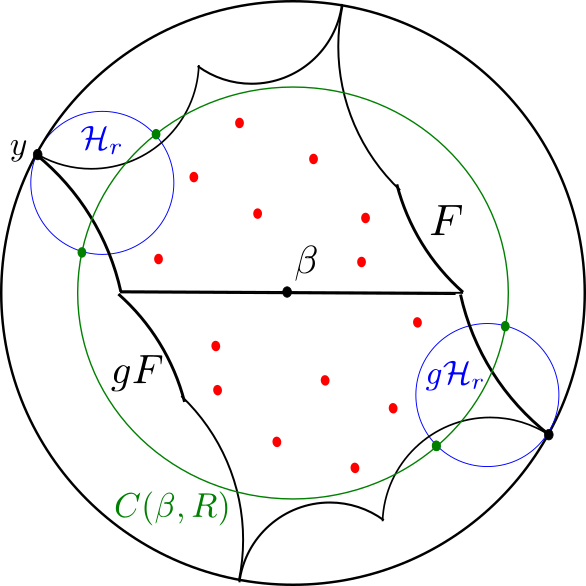}
      \caption{$\text{Aut}(M)$ contains no parabolic elements.}
      \label{fig:no_parabolic}
    \end{minipage}
\end{figure}

\begin{proposition}\label{lem:G_hyp_fin_gen}
    Let $M$ be a locally-finite map in the hyperbolic plane with 1-end, such that $V(M)$ is a locally finite collection of vertices. Suppose a group $G\leq \text{Isom}(\HH^2)$ acts quasi-transitively on $M$. Then $G$ is finitely generated.
\end{proposition}
\begin{proof}
We again assume that $G$ is Fuchsian. If not, then there exists a Fuchsian subgroup $G^+\leq G$ of index 2, and it is known that if a group contains a finitely generated subgroup of finite index then it is itself finitely generated.
  A Fuchsian group $G$ such that all its fundamental domains have finite area is cocompact if and only if it contains no parabolic elements (\cite{Fuchsian_Katok}, 4.2.2). Suppose toward the contrary that $G$ has a parabolic element, then any fundamental domain $F$ for $G$ has an ideal vertex $y\in\partial \mathbb{D}$ that is the fixed point of a parabolic element $p\in G$. 
There exists at least one side $S$ of $F$ not incident with the ideal vertex $y$, since if not $F$ has empty interior. Then there exists side $S'$ of $F$ which does not contain $x$ and is paired to $S$ by an element $g\in G$ so that $gS=S'$ and $F\cap gF=S'$. Define $E=F\cup gF$ and let $\beta\in S'$ so that $\beta$ is in the interior of $E$. Define geodesic arcs $\Sigma,\Sigma':[0,\infty)\rightarrow \overline{\mathbb{D}}$ such that $\Sigma(0)=\beta=\Sigma'(0)$, $\Sigma(\infty)=y$, $\Sigma'(\infty)=gy$ and so that $\rho_\mathbb{D}(\beta,\Sigma(t))=\rho_\mathbb{D}(\beta,\Sigma'(t))=t$ for all $t>0$.
 Then there exists $r>0$ such that $B(\beta,r)\cap F\cap V(M)=F\cap V(M)$. 
  Define $\omega(t)$ to be the horocycle intersecting $\partial\mathbb{D}$ at $x$ such that $\omega(t)\cap\Sigma=\Sigma(t)$. Then as $\omega(t)$ is $p$-invariant for each $t>0$, we must have that the closed horocyclic region $\mathcal{H}_r$ defined by $\omega(r)$ contains no vertices of $M$ (see Figure \ref{fig:no_parabolic}). 
 Then there exists $R>0$ sufficiently large that $C(\beta,R)$ intersects $\mathcal{H}_r$ in two points $\zeta,\zeta'\in\mathbb{D}$, and $g\mathcal{H}_r$ in two points $z,z'\in\mathbb{D}$. Choose $R>\max\{\mathcal{E}\}$ sufficiently large so that 
  $$\min\{\rho_\mathbb{D}(y,y'),\rho_\mathbb{D}(z,z')\}>\max\{\mathcal{E}\}.$$
  Then $\mathbb{D}-(\overline{B}(\beta,R)\cup \mathcal{H}_r\cup g\mathcal{H}_r)$ consists of two infinite connected components $P_1$ and $P_2$ and the induced subgraphs of $M$ on $P_1$ and $P_2$ are infinite and connected. Moreover, if $u\in V(M)\cap P_1$ and $v\in V(M)\cap P_2$ then there cannot be an edge between them. Then $\overline{B}(\beta,R)\cap V(M)$ is a finite cutset separating $M$ into two infinite components, contradicting that $M$ has one end.\\

\end{proof}

The following theorem of Edmonds, Ewing and Kulkarni concerning the precise indices of torsion free subgroups for finitely generated Fuchsian groups will be necessary for our argument. It provides necessary and sufficient conditions on the possible indices of torsion free subgroups.
The existence of finite index torsion free subgroups is a classical result of Bundgaard, Nielsen \cite{Bungaard_Nielsen_51} and Fox \cite{Fox_torsion_free} in their proof of Fenchel's conjecture.
As a consequence we then have torsion free subgroups with arbitrarily large index, and hence subgroups with fundamental domains of arbitrarily large area. 
\begin{theorem}[\cite{Edmonds1982}]\label{thm:EEK_torsion}
    Let $G$ be an infinite, finitely generated Fuchsian group. Then there exists a constant $k>0$ such that there exists a torsion free subgroup of finite index $n$ if and only if $n$ is divisible by $k$.
\end{theorem}

  The version of the theorem written above is a simplified version of the statement as it is all we need for the purposes of this paper. However, one may note that the proof itself is constructive and an explicit form for $k$ is given that depends only on the list of maximal finite orders of elements in $G$.

\begin{theorem}\label{thm:main_Hyp}
    Let $M$ be a locally-finite map in the hyperbolic plane with 1-end, such that $V(M)$ is a locally finite collection of vertices. Suppose that $G\leq\text{Isom}(\HH^2)$ acts quasi-transitively on $M$. Then there exists a proper vertex colouring $\mathcal{C}:V(M)\rightarrow \{1,\ldots,m\}$, and a finite index subgoup $T\leq G$, such that $\mathcal{C}$ is $T$-invariant.
\end{theorem}
\begin{proof}
    It immediately follows that exists an orientation-preserving Fuchsian subgroup $G^+\leq G$ of index at most 2, acting quasi-transitively on $M$. It follows by Theorem \ref{thm:EEK_torsion} that there exists a finitely generated minimal index torsion free subgroup $T_0\leq G^+\leq G$ of index $[G:T_0]=k$. Let $\text{Tr}(\gamma)$ denote the translation distance of an element $\gamma\in G$, and $\text{Tr}(H)$ denote the infimum translation distance of a subgroup $H\leq G$. Since $T_0$ is discrete and contains only hyperbolic elements it follows that $\text{Tr}(T_0)$ is attained by some element $t\in T_0$ so that $\text{Tr}(t)=\text{Tr}(T_0)$. Furthermore, the translation lengths $\{\text{Tr}(t):t\in T_0\}$ forms a discrete set with finitely many elements attaining any given translation length. It then follows that we can write elements in a sequence of ascending translation lengths $T_0=\{t_1,t_2,\ldots\}$ such that $\text{Tr}(t_1)\leq \text{Tr}(t_2)\leq \ldots$\;. Then for a given translation length $L>0$ the set $A_L:=\{t_i: \text{Tr}(t_i)<L\}$ is finite, and can be written $A_L=\{t_1,\ldots,t_{N}\}$ for some $N>0$. Fix $L>0$ such that $L>\max\{\mathcal{E}\}$.
    Since Fuchsian groups are residually finite, for each $1\leq i\leq N$ there exists an finite index normal subgroup $T_i\leq T_0$ such that $t_i\not\in T_i$. Let $T=\bigcap_{i=1}^{N}T_i$. Then as a finite intersection of finite index subgroups $T$ is finite index in $T_0$, and by construction there are no edges $\{u,v\}\in E(M)$ such that $u,v$ are both in the same $T$-orbit. 
    Let $\pi:\HH^2\rightarrow \HH^2/T$ be the quotient map. It follows the induced quotient graph $M/T$ embedded on $\HH^2/T$ is finite, and contains no loops. Let $m:=|V(M/T)|$. Define a colouring $\widetilde{\mathcal{C}}:V(M/T)\rightarrow \{1,\ldots,m\}$ by colouring each vertex uniquely. Then this lifts to a proper colouring $\mathcal{C}:V(M)\rightarrow \{1,\ldots,m\}$ that is $T$-invariant.\\
\end{proof}

\begin{remark}\label{rem:index_colours}
One notes that the number of colours $m$ of a periodic colouring constructed as above is directly proportional to the index of the subgroup $T$. Since $|V(M/T)|=[G:T]\cdot |V(M/G)|$.
\end{remark}
\begin{remark}\label{rem:LERF}
    Fuchsian groups satisify the stronger property of being \textit{locally extended residually finite} (LERF) \cite{Lopez94}. This has the consequence that there exist finite index torsion free subgroups $T_i\leq T_0$, so that $t_i\notin T_i$, for all $1\leq i\leq N$, that are not necessarily normal in $T_0$.
\end{remark}

\begin{corollary}
    Let $\Gamma$ be a locally-finite, quasi-transitive, 3-connected planar graph with 1-end. Then $\Gamma$ permits a periodic proper vertex colouring.
\end{corollary}
\begin{proof}
    By Theorem \ref{thm:Babai quasi-transitive} $\Gamma$ embeds into one of the natural geometries $\EE^2$ or $\HH^2$ as a locally-finite map $M$ with 1-end such that the automorphisms of $M$ are induced by isometries in the geometry. It follows from Lemma \ref{lem:locally finite collection} that $V(M)$ is a locally finite collection of vertices. Then $\text{Aut}(\Gamma)\cong\text{Aut}(M)=G\leq \text{Isom}(\mathcal{X})$, for $\mathcal{X}\in\{\EE^2,\HH^2\}$. Then $G$ acts quasi-transitively on $M$. The result then follows from Theorems \ref{thm:main_Euclid} and $\ref{thm:main_Hyp}$.\\
\end{proof}

\section{Graphs that are not 3-connected}\label{sec:not_3-con}
We turn now to the case that $\Gamma$ is a locally-finite, planar connected graph with 1-end that is not 3-connected.
A \textit{subdivision} of a graph $X$ is a graph obtained by adding finitely many vertices of degree 2 to each edge of $X$. A \textit{topological subgraph} $Y$ of a graph $X$ is a graph such that some subdivision of $Y$ is a subgraph of $X$. \\
To analyse the $\Gamma$ that are not 3-connected we will make use of known reduction technique first introduced by Mac Lane \cite{MacLane_Reduction}, and used in \cite{Babai1975}, \cite{Babai97}, \cite{Hopcroft1972} among many others\footnote{See \cite{KlavNedZeman2021} for further discussion and references on applications of this reduction technique.}. 
This technique produces a topological subgraph of $\Gamma$ that is 3-connected and well behaved under the inherited action of the automorphisms of $\Gamma$. 

Though this technique is described in detail in the above mentioned papers we will outline it here for completeness and ease of reference.
First we introduce a number of notions; for a subset $U\subseteq V(\Gamma)$ the \textit{boundary} $\partial U$ of $U$ is the set of vertices not contained in $U$ that are adjacent to a vertex in $U$. If $\Gamma$ is $\kappa$-connected, and $F\subsetneq V(\Gamma)$ a proper subset such that $F$ induces a connected subgraph, then it follows that $|\partial F|\geq \kappa$. If $|\partial F|=\kappa$ then $F$ is called a \textit{fragment} and the set $C=\partial F$ a \textit{minimal cutset}. An inclusion minimal fragment is called an \textit{atom}. That is, a fragment $F$ is an atom if and only if there are no fragments $F'$ such that $F'\subsetneq F.$ 
\\ 
The reduction constructs a sequence of graphs $\Gamma=\Gamma_0,\Gamma_1,\ldots, \Gamma_m$, so that $\Gamma_{i+1}$ is a topological subgraph of $\Gamma_i$ and; 
\begin{enumerate}
    \item each $\Gamma_i$ is connected,
    \item each $\Gamma_i$ has 1-end,
    \item $V(\Gamma_i)\supseteq V(\Gamma_{i+1})$, and
    \item $\text{Aut}(\Gamma_i)$ acts quasi-transitively on $V(\Gamma_{i+1})$
\end{enumerate}
(see \cite[Theorem 4.1]{Babai97}). The reduction from $\Gamma_i$ to $\Gamma_{i+1}$ goes as follows; We first assume that $\Gamma_i$ is 2-connected.
For each atom $A\subset \Gamma_i$ remove the vertices and incident edges of $V(A)$, and replace them by an edge between the two vertices of $\partial A$. If $\Gamma_i$ is 1-connected but not 2-connected, then delete all atoms to obtain $\Gamma_{i+1}$. Let $G:=\text{Aut}(\Gamma)$ and $G_i:=\text{Aut}(\Gamma_i)$ for $i\leq m$.
In addition to (iv) above, we may further conclude from the inherited structure of each $\Gamma_{i+1}$ that, the induced action of $G_i$ on $V(\Gamma_{i+1})$ is via graph automorphisms, implies there exists some subgroup of $G_{i+1}$ that acts freely on $V(\Gamma_{i+1})$, and with equivalent action to that of $G_i$ on $V(\Gamma_i+1)$.
Moreover, we note that given any subgroup $H$ of $G$, we have the inherited action on $V(\Gamma)$ and in turn this defines an induced action by graph automorphisms on $V(\Gamma_i)$ for all $i\leq m$.

It is known that locally-finite, planar, vertex transitive graphs with 1-end are 3-connected \cite[Theorem 1]{Babai1975}. Hence this reduction sequence is guaranteed to conclude in a 3-connected graph $\Gamma_m$ in $m<\infty$ steps, as the number of vertex orbits strictly reduces with each step. As $\Gamma_m$ is a locally-finite, quasi-transitive, 3-connected planar graph with 1-end there exists an embedded map $M_m$ in $\mathcal{X}$ such that all automorphisms are induced by isometries of $\mathcal{X}$. As such $G_m$ can be identified with a group of discrete isometries of $\mathcal{X}$.

\begin{definition}
For $g,h\in G$, we say that $g\sim h$ if and only if $g(v)=h(v)$ for all $v\in V(\Gamma_m)\subset V(\Gamma)$. If $g\sim h$ we may write that $g\in[h]$.
\end{definition}
Let $\overline{G}$ denote the quotient group $G/\sim$.
\begin{lemma}\label{lem:subgroup_of_G}
    There exists a subgroup of $G$ that can be identified with $\overline{G}$.
\end{lemma}
\begin{proof}
Let $\{1,s_1,s_2,s_3,\ldots\}$ be a set of equivalence class representatives so that $g\in[s_i]$ for some $i\geq 0$ for all $g\in G$ ($s_0=1$). Define a function $\varphi:G\rightarrow G$ such that, if $g\in[s_i]$ then $\varphi(g)=s_i$. Then for $g,h$ in $[s_i],[s_j]$ respectively, we have that for all $u\in V(\Gamma_m)\subset V(\Gamma)$, 
$$(\varphi(g)\varphi(h))(u)=\varphi(g)(\varphi(h)(u))=s_i(s_j(u))=(s_is_j)(u)=\varphi(gh)(u).$$
Hence $\varphi$ is a homomorphism, and $\text{Ker}(\varphi)=[1]$ is a normal subgroup of $G$.
   The result then follows by the first isomorphism theorem, with $\overline{G}:=\text{Im}(\varphi)$.\\
\end{proof}

It is helpful now to highlight a consequence of previous work. For $\mathcal{X}\in\{\EE^2,\HH^2\}$, let $\mu_\mathcal{X}(\cdot)$ denote the $\mathcal{X}$-area of a connected subset. 
\begin{lemma}\label{lem:fin_area_action_only}
    Let $\Gamma$ be a locally-finite, 3-connected, quasi-transitive, planar graph with 1-end with embedding $M$ in $\mathcal{X}$($=\EE^2$ or $\HH^2$). Suppose that $G\leq \text{Isom}(\mathcal{X})$ is a group acting quasi-transitively on $V(M)$, then $G$ is a discrete group of isometries such that if $F\subset\mathcal{X}$ is a fundamental domain for $G$ then $\mu_\mathcal{X}(F)<\infty$.
\end{lemma}
\begin{proof}
  Follows from Lemmas \ref{lem:wallpaper} and \ref{lem:hyp_fin_area}.\\
\end{proof}

\begin{lemma}\label{lem:G_bar_acts_quasi-transitive_Gamma_m}
    There exists a subgroup of $G_m$ that can be identified with $\overline{G}$. In particular one may consider $\overline{G}$ as a discrete group of isometries of $\mathcal{X}$ that acts quasi-transitively on $M_m$.
\end{lemma}
\begin{proof}
There is a natural induced action of $G$ on $V(\Gamma_m)$ as a subset of $V(\Gamma)$. By definition we have that the action of $G$ on $V(\Gamma_m)$ is equivalent to the free action of $G/\sim$ on $V(\Gamma_m)$.
Moreover, by the inherited structure it follows that $G/\sim$ acts on $\Gamma_m$ via graph automorphisms and, since $G/\sim$ is a group closed under composition, it follows there is an isomorphic subgroup $\widetilde{G}\leq G_m$. Hence under the embedding of Theorem \ref{thm:Babai quasi-transitive} we may consider $\widetilde{G}$ as a subgroup of isometries of $\mathcal{X}$.
Since $G$ acts quasi-transitively on $V(\Gamma)\supset V(\Gamma_m)$ it follows that $G/\sim$ acts quasi-transitively on $V(\Gamma_m)$, and hence $\widetilde{G}$ acts quasi-transitively on $V(M_m)$.
\\
\end{proof}

\begin{theorem}\label{thm:big_thm_graph_permit_col}
    Let $\Gamma$ be a locally-finite, quasi-transitive planar graph with 1-end. Then $\Gamma$ permits a periodic vertex colouring.
\end{theorem}
\begin{proof}
    The case that $\Gamma$ is 3-connected has been shown in Theorems \ref{thm:main_Euclid} and \ref{thm:main_Hyp}. Assume that $\Gamma$ is not 3-connected. Then there exists a topological subgraph $\Gamma_m$ that is locally-finite, planar with 1-end and is 3-connected. Under Theorem \ref{thm:Babai quasi-transitive} let $M_m$ be its embedding into a natural geometry $\mathcal{X}\in\{\EE^2,\HH^2\}$ so that $\text{Aut}(\Gamma_m)=\text{Aut}(M_m)=:G_m$ can be identified with a discrete group of isometries of $\mathcal{X}$.\\
    We proceed by using $M_m$ to construct an embedding $M$ of the full graph $\Gamma$, such that $\overline{G}\leq G_m\leq \text{Isom}(\mathcal{X})$ acts quasi-transitively on $M$. Then we may apply Theorems \ref{thm:main_Euclid} and \ref{thm:main_Hyp} to find a finite index subgroup $T\leq \overline{G}$ and a vertex colouring $\mathcal{C}$ that is $T$-invariant.
    Then as $T\leq\overline{G}\leq G=\text{Aut}(\Gamma)$ (Lemma \ref{lem:subgroup_of_G}), we have found a periodic colouring of $\Gamma$.

    To construct the embedded map $M$ of $\Gamma$ from $M_m$ we shall reintroduce the removed atoms into the surface. To do this we construct a sequence of maps $M_m,M_{m-1},\ldots, M_1,M$ for $\Gamma_m,\Gamma_{m-1},\ldots, \Gamma_1,\Gamma_0=\Gamma$ respectively so that $\overline{G}$ acts quasi-transitively on $M_i$ by isometries, for each $i\leq m$. 

    For each $\Gamma_i$ let $\mathcal{A}_i$ be the set of atoms in $\Gamma_i$ that are removed by the reduction to $\Gamma_{i+1}$.
    To construct $M_{i-1}$ from $M_i$ we need to reintroduce all the atoms in $\mathcal{A}_{i-1}$. We make the initial assumption that $\overline{G}$ acts quasi-transitively on $M_i$ as a group of $\mathcal{X}$-isometries and $V(M_i)$ is a locally finite collection of vertices. Since this is true for the case of $M_m$, the result will then follow inductively. Moreover as $\Gamma_i$ is locally-finite it follows that each $M_i$ is locally-finite. 
    We split into two cases;
    \begin{itemize}
        \item Suppose that $\Gamma_{i-1}$ is 1-connected but not 2-connected. Let $\mathcal{V}:=\{v\in V(M_i):v\in\partial A, A\in\mathcal{A}_{i-1}\}$. Then we may partition $\mathcal{V}$ by its $\overline{G}$ orbits so that $\mathcal{V}=\overline{G}v_1\sqcup \ldots \sqcup \overline{G}v_s$, for some set of orbit representatives $v_j, j\leq s$, and $s<\infty$.
        For all $A\in\mathcal{A}_{i-1}$ we have that $\partial A=\{u\}$ for some vertex $u\in \mathcal{V}$.
        Define $\mathcal{A}_{i-1}^{j}$ to be the set of atoms in $\mathcal{A}_{i-1}$ so that $\partial A\in \overline{ G}v_j$. Then we have the partition $\mathcal{A}_{i-1}=\mathcal{A}_{i-1}^{1}\sqcup \ldots \sqcup \mathcal{A}_{i-1}^s$. Moreover, since $\overline{G}\leq G$ it follows that for all $j\leq s$ that if $A,A'\in \mathcal{A}_{i-1}^j$ then $A$ and $A$ are isomorphic as subgraphs of $\Gamma$.
       Since $V(M_i)$ a locally finite collection of vertices, and the degree of each vertex in $V(M_i)$ is finite, at each $v_j$, $j\leq s$, we can find a closed connected set $Y_i\subset \mathcal{X}$ with non-empty interior such that $Y_j\cap (V(M_i)\cup E(M_i))=\{v_j\}$. 
       Since $\Gamma$ is planar, all its subgraphs are planar, and for $A\in\mathcal{A}_{i-1}^j$ there exists an embedding $\iota: A\rightarrow \mathring{Y_j}$, for all $j\leq s$. Additionally since each $A\in \mathcal{A}_{i-1}$ is finite, we may construct an edge curve $\alpha:[0,1]\rightarrow Y_j$ so that $\alpha(0)=\partial A= v_j$, $\alpha(1)$ is the appropriate adjoining vertex in $\iota(A)$, and so that $\alpha \cap (E(M_i)\cup E(\iota(A))=\emptyset$. We then repeat this embedding and edge arc isometrically for each $v\in\overline{G}v_j$ and $j\leq s$ to construct $M_{i-1}$. 
       It is then clear that $\overline{G}$ acts quasi-transitively on $M_{i-1}$ as a group of planar isometries. The fact that $V(M_{i-1})$ is a locally finite collection of vertices follows since any compact set can contain at most finitely mant vertices of $\mathcal{V}$, and within a bounded neighbourhood of each vertex of $\mathcal{V}$ only a finite number
       of new vertices have been added, so that the set $K\cap (V(M_{i-1})\setminus V(M_i))$ is finite for all compact $K\subset\mathcal{X}$.

        \item Suppose that $\Gamma_{i-1}$ was 2-connected, then let $\mathcal{E}_i:=\{\{u,v\}\in E(M_i): \{u,v\}\in E(\Gamma_i)\setminus E(\Gamma_{i-1})\}$. Each edge in $\{u,v\}\in E(M_i)$ is a path $\alpha_{u,v}:[0,1]\rightarrow \mathcal{X}$ from $u$ to $v$, so that if $\alpha_{u',v'}$ is a distinct edge path, then $\alpha_{u,v}\cap\alpha_{u',v'}=\emptyset$. 
        Without loss of generality suppose each $\alpha_{u,v}$ is parametrised so that the arc length from $\alpha_{u,v}(0)=u$ to $\alpha_{u,v}(t)$ is $t$ times the total path length. By Lemma \ref{lem:G_bar_acts_quasi-transitive_Gamma_m} $\overline{G}\leq G_m$ and there is an induced action by isometries of $\overline{G}$ on $\mathcal{X}$ that is quasi-transitive on $V(M_m)$. By previous steps we assume that the isometric action of $\overline{G}\leq \text{Isom}(\mathcal{X})$ is quasi-transitive on $M_i$. It is clear then that $\overline{G}$ acts quasi-transitively on $E(M_i)$, so that we may partition $\mathcal{E}_i=\overline{G}e_1\sqcup \ldots \sqcup \overline{G}e_s$, where $e_j=(u_j,v_j)\in E(M_i).$
        Furthermore, since $V(M_i)$ is a locally finite collection of vertices, and the edges fall into finitely many isometry classes (under action of $\overline{G}\leq \text{Isom}(\mathcal{X})$), we can conclude that the set of all edges of $M_i$ is also a locally finite collection of vertices in $\mathcal{X}$.
        As above we define sets $\mathcal{A}_i^j\subset\mathcal{A}_i$ so that $A\in\mathcal{A}_i^j$ if and only if $\partial A=\{u,v\}$ so that the edge $e=\{u,v\}\in \overline{G}e_j$. Then $ \mathcal{A}_i=\mathcal{A}_i^1\sqcup \ldots \sqcup\mathcal{A}_i^s $. Moreover, we note that since $\overline{G}\leq G_i$, for any two atoms $A,A'\in\mathcal{A}_i^j$ there exists $g\in G_i$ so that $gA=A'$, hence are isomorphic as subgraphs.
       Since the edge and vertex sets for $M_i$ are locally finite collection of vertices, for each $ge=(gu,gv)\in \overline{G}e_j$ we can find $r_j>0$ such that the the closed ball $\overline{B}(g\alpha_{u,v}(0.5),r_j)\cap V(M_i)=\emptyset=\overline{B}(g\alpha_{u,v}(0.5),r_j)\cap E(M_i)$.
        Since all $A\in \mathcal{A}_i^j$ are isomorphic and planar, for all $A\in\mathcal{A}_i^j$ we define an embedding $\iota_j:A\rightarrow \overline{B}(gx_j,r_j)$, and an edge path $\alpha:[0,1]\rightarrow \mathcal{X}$ from $A$ to $\{gv_j,gu_j\}=\partial A$, such that if $hA=A'$, $A,A'\in\mathcal{A}_i^j$  ,$h\in\overline{G}$, then the corresponding isometric action on $\mathcal{X}$ gives $h\iota_j(A)=\iota_j(A')$.
        Do this simultaneously for all $A\in\mathcal{A}_i$ to obtain a $M_{i-1}$ that is acted on quasi-transitively by $\overline{G}\leq \text{Isom}(\mathcal{X}).$ Following a similar argument to the above case we also have that $V(M_{i-1})$ is a locally finite collection of vertices.
    \end{itemize}
    Hence we may find an embedding $M$ for $\Gamma$ so that there is an quasi-transitive action on $M$ by $\overline{G}\leq \text{Isom}(\mathcal{X})$, and the result follows.\\ 
\end{proof}
In fact, in the Euclidean case we have actually proved a stronger result.
\begin{theorem}
    Let $\Gamma$ be a locally-finite quasi-transitive planar graph with 1-end, such that its 3-connected topological subgraph $\Gamma_m$ embeds into the Euclidean plane $\EE^2$ under Theorem \ref{thm:Babai quasi-transitive}. Then $\Gamma$ permits a periodic 5-colouring. 
\end{theorem}

\section{Discussion and open problems}\label{sec:discussion}

\textbf{Edge colourings and orientations.} In addition to periodic vertex colourings \cite{abrishami2025periodiccoloringsorientationsinfinite} also discusses the existence of periodic \textit{edge colourings} and \textit{orientations}. For brevity we do not include the precise definitions here, and direct the interested reader to the article in question. 
Observe that if there exists a periodic vertex colouring for a given graph, a periodic orientation can be defined by any total ordering on the set of colours. Together with the work in \cite{abrishami2025periodiccoloringsorientationsinfinite} we may conclude:
\begin{corollary}
    Let $\Gamma$ be a locally-finite, quasi-transitive planar graph with finitely many ends. Then $\Gamma$ permits a periodic orientation.
\end{corollary}
For the existence of periodic edge colourings, we would like to follow a similar argument as \cite[Corollary 5.5]{abrishami2025periodiccoloringsorientationsinfinite} (the analogous result for 2-ended graphs) and prove existence by passing to the \textit{line graph}. However, since the line graph of a planar graph is not necessarily planar there are restrictions on this approach.
We have the following two results of Sedláček, and Greenwell and Hemminger that give two equivalent conditions for a graph to have a planar line graph. 
For two graphs $\Gamma_1,\Gamma_2$ we
define the \textit{join} $\Gamma_1+\Gamma_2$ to be graph with vertex set $V(\Gamma_1+\Gamma_2)=V(\Gamma_1)\cup V(\Gamma_2)$, and edge set $E(\Gamma_1+\Gamma_2)=E(\Gamma_1)\cup E(\Gamma_2)\cup \{\{u,v\}:u\in V(\Gamma_1), v\in V(\Gamma_2)\}$.

 \begin{theorem}[\cite{GREENWELL197231}]\label{thm:line_2}
    A graph $\Gamma$ has a line graph $L(\Gamma)$ that is planar if and only if it has no subgraph isomorphic to $K_{3,3},\,K_{1,5},\,P_4+K_1$ or $K_2+\Bar{K}_3$.
\end{theorem}
\begin{theorem}[\cite{Sedla_line_graph}]\label{thm:line_1}
    A graph $\Gamma$ has a line graph $L(\Gamma)$ that is planar if and only if the maximum degree of any vertex is 4, and each vertex of degree 4 is a cut vertex.
\end{theorem}
\begin{corollary}
    If $\Gamma$ is a locally finite, quasi-transitive planar graph with 1-end satisfying the hypotheses of Theorems \ref{thm:line_1} or \ref{thm:line_2}, then there exists a periodic edge colouring of $\Gamma$.
\end{corollary}
\begin{proof}
    Observe that the line graph $L(\Gamma)$ is also locally finite, connected and quasi-transitive, and by our hypotheses must be planar. In addition the rays of $\Gamma$ and $L(\Gamma)$ are in correspondence, and similarly given any pair of rays $r,r'$ in $\Gamma$ and their respective line graph rays $L(r), L(r')$ in $L(\Gamma)$, the infinite set of disjoint rays between vertices of $r$ and vertices of $r'$ are in correspondence with an infinite set of disjoint rays between vertices of $L(r)$ and vertices of $L(r')$. Hence we may conclude that $L(\Gamma)$ also has 1-end. Thus we conclude by applying Theorem \ref{thm:big_thm_graph_permit_col} to $L(\Gamma)$.\\
\end{proof}
It is certainly conceivable an alternative construction of a periodic edge colouring exists that is not dependent on the planarity of the line graph. For example by following similar techniques employed in this paper but focussing on the edges. This leaves the following open problem.
\begin{problem}
    For a locally-finite, quasi-transitive planar graph $\Gamma$ with 1-end is it possible to construct a periodic edge colouring without requiring the planarity of the line graph?
\end{problem}

\textbf{Constructibility of colourings.} Of course for any existence statement as we have proved above, and in particular for graph colourings, one should comment on the constructibility of said colourings. Suppose we have a graph $\Gamma$ and its automorphism group $G$. The first stage to examine is Babai's automorphism respecting embeddings; for a given locally-finite, quasi-transitive 3-connected graph with 1-end the construction laid out in \cite{Babai97} can be reduced to assigning an automorphism respecting rotation system to each of the vertices, and the geometry is then determined by the curvature of the induced 2-dimensional cell complex. The process of finding a maximal torsion free subgroup of a 2-dimensional discrete Euclidean or hyperbolic isometry group, is well studied, and not hard to identify up to its \textit{signature}\footnote{either Fuchsian signature (as in \cite{Singerman1970}, \cite{Fuchsian_Katok}) or Conway's orbifold signature (as in \cite{conway2008symmetries})} and isomorphism class (see \cite{Edmonds1982},\cite{Schattschneider}). 
In the Euclidean case, given the maximum edge length it is not hard to derive a subgroup with sufficient minimum translation length in order that the quotient map satisfies the hypotheses of Thomassen's theorem. However we are not able to define an explicit colouring from Thomassen's theorem, the proof is a non-constructive proof-by-contradiction and relies in part on the planar four colour map theorem.

In the hyperbolic case, we rely on the residual finiteness (or alternatively the stronger property of being locally extended residually finite or LERF) of Fuchsian groups to infer the existence of normal (or not necessarily normal) subgroups of sufficiently large minimum translation length. 
Whichever property we use, the desired subgroups are found as finite index subgroups of surface groups (which are themselves finite index in Fuchsian groups). Although constructing each subgroup $T_i$, not containing some $t_i$, would be a long path to tread, and require some translation of hyperbolic isometries to elements of the fundamental group of a quotient surface, the proofs provided in \cite{Hempel_72} and \cite{Lopez94} describe how to construct subgroups not containing a chosen element. 
However the process to identify the full list of all the elements to remove, that is elements whose translation lengths are too small, is not clear to the author. \\ 
Note that a similar approach can be applied to construct explicit periodic colourings in the Euclidean case, though of course with the number of colours depending on the index of the translation subgroup.
Finally, for graphs that are not 3-connected, the reduction technique to a 3-connected topological subgraph is explicitly constructive, with the algorithm provided as its definition in \S\ref{sec:not_3-con}.


\textbf{Reducing the number of colours in the hyperbolic case.} The application of Thomassen's Theorem \ref{thm:Thomassen} in the Euclidean case is remarkably powerful here, providing a uniform bound on the chromatic number of all such graphs.
In the hyperbolic case, we are not so fortunate. By relying on residual finiteness to construct our desired finite index subgroups it means that we have very little control over the size of the index, and consequently over the number of colours (see Remark \ref{rem:index_colours}). 
For any integer $B>0$ it is not hard to construct maps in the hyperbolic plane with the property that any subgroup with sufficiently large minimum translation length has index at least $B$. Then in contrast with the Euclidean case, for an arbitrarily large integer $C>0$ there exist graphs that embed into the hyperbolic plane such that the periodic colouring defined by the methods above uses at least $C$ colours, for arbitrarily large $C$.
As such it is desirable to be able to apply Theorem \ref{thm:Thomassen} to those graphs that embed into the hyperbolic plane. However constructing an embedding $M$ and a subgroup $H\leq \text{Aut}(M)$ so that the quotient map $M/H$ on the quotient surface $\HH^2/H$ has non-contractible cycles of sufficient size is unclear.
In the Euclidean case, although $2^{14g+6}$ is large, whenever we take a torsion free wallpaper group $T\leq \text{Isom}(\EE^2)$ the quotient surface is always a torus with genus $g=1$. This means one can continue to take larger index translation subgroups, and the genus of the quotient surface always stays constant. This fixes a minimum non-contractible cycle size to aim for, then we can take successively large index subgroups until the minimum translation distance is sufficient.
However in the hyperbolic plane this is never the case, as the genus of an index $n$ subgroup is determined by the Riemann-Hurwitz formula (see \cite{Singerman1970} for example).
That is, for a locally-finite, quasi-transitive map $M$ with 1-end in the hyperbolic plane, so that $V(M)$ is a locally finite collection and a Fuchsian group $G^+\leq\text{Isom}(\HH^2)$ acts quasi-transitively on $M$, if $T\leq G^+$ is a finite index torsion free subgroup of index $N$ then the genus $g_T$ of the quotient surface $\HH^2/T$ is given by
\begin{align}\label{eq:g_t}
    g_T=\frac{N}{2}\cdot\left[2g-2+\sum_{i=1}^r\left( 1-\frac{1}{m_i}\right)\right]+1
\end{align}
where $g$ is the genus of $\HH^2/G^+$, and $\{m_i\}_{i\leq r}$ is the set of finite orders of the elliptic generators of $G^+$.
As the expression inside the square brackets describes $1/2\pi$ times the area of a fundamental domain for $G^+$, which is always positive, it follows that the genus $g_T$ grows linearly and monotonically with the index.
This has the consequence that, each time we increase the index and pass to a further subgroup to remove small translation lengths, we are also increasing the genus of the subgroup quotient surface linearly. Then in turn this exponentially increases the minimum non-contractible cycle size required of the quotient map for Theorem \ref{thm:Thomassen} to be applied. The discrepancy of the growth between these two variables ensures that under the methods above there is no general construction that will force a quotient map to satisfy the hypotheses of Theorem \ref{thm:Thomassen}.  
Although this approach is unfruitful, the Euclidean results still suggest the following problem. 
\begin{problem}
    Does there exist a fixed integer $A$ such that all locally-finite, quasi-transitive planar graphs with 1-end, whose 3-connected topological subgraph embeds into the hyperbolic plane under Theorem \ref{thm:Babai quasi-transitive}, can be periodically coloured in $A$ colours?
\end{problem}
There is some additional control over the number of colours used in our construction that is permitted by the Ringel-Youngs theorem \cite{RingelYoungs68}, which states that for any map embedded on an orientable surface of genus $g>0$ there exists a proper vertex colouring with $\lfloor (7+\sqrt{1+48g})/2\rfloor$ colours. When colouring the quotient map $M/T$ in the proof of Theorem \ref{thm:main_Hyp} it follows then we can reduce the number of colours used from $|V(\Gamma/T)|$ colours to $\lfloor( 7+\sqrt{1+48g_T})/2\rfloor$
colours, where $g_T$ can be obtained from (\ref{eq:g_t}). 
Since $g_T$ is dependent only on $G$ and the maximum edge length of the embedding $M$ of $\Gamma$, one avenue to establish a lower bound for all graphs that embed into the hyperbolic plane would be by answering the following problem for hyperbolic isometry groups.  
\begin{problem}
   For a finitely generated discrete group of hyperbolic isometries $G$ and a distance $d>0$, what is the minimal index of a torsion free subgroup $T$ such that $T$ contains no elements whose translation length is less than or equal to $d$?
\end{problem}

\textbf{Acknowledgements.} The author would like to thank their supervisor James Anderson for the many helpful comments and discussions.  




\bibliography{biblography}

\end{document}

%% file: Preamble.tex
\usepackage{graphicx} 
\usepackage{amsmath}  
\usepackage{amsfonts} 
\usepackage{enumitem}
\usepackage[parfill]{parskip}

\usepackage[scale=0.8,vmarginratio={1:2},heightrounded]{geometry}

\usepackage{amsthm} 
\usepackage[affil-it]{authblk} 
\usepackage{amsopn} 
\usepackage{amssymb} 
\usepackage{mathrsfs} 
\usepackage{booktabs} 
\usepackage{natbib} 
\usepackage{tikz} 
\usepackage{algorithm} 
\usepackage{float} 
\usepackage{caption} 
\usepackage{abstract} 
\usepackage{listings} 
\usepackage{color} 

\usepackage[pagebackref]{hyperref} 
\usepackage{pdfpages}

\newfloat{algorithm}{t}{lop}

\numberwithin{equation}{section}

\newcommand{\ZZ}{\mathbb{Z}}

\newcommand{\EE}{\mathbb{E}}
\newcommand{\CC}{\mathbb{C}}

\newcommand{\HH}{\mathbb{H}}

\newtheorem{lemma}{Lemma}[section]
\newtheorem{theorem}[lemma]{Theorem}
\newtheorem{proposition}[lemma]{Proposition}
\newtheorem{corollary}[lemma]{Corollary}
\newtheorem{problem}{Problem}[section]

\theoremstyle{definition}
\newtheorem{definition}[lemma]{Definition}

\theoremstyle{remark}
\newtheorem{remark}{Remark}[section]

\DeclareSymbolFont{bbold}{U}{bbold}{m}{n}
\DeclareSymbolFontAlphabet{\mathbbold}{bbold}